\documentclass{amsart}

\usepackage{tikz-cd}
\usepackage[hyphens]{url}

\usepackage{amsthm,amsmath,amssymb,amsfonts,dsfont}
\usepackage[colorlinks, linkcolor=blue, citecolor=blue]{hyperref}
\usepackage[indentafter]{titlesec}
\titleformat{name=\section}{}{\thetitle.}{0.8em}{\centering\scshape}
\usepackage{mathrsfs}
\usepackage{bm}
\usepackage{titlesec}
\numberwithin{equation}{section}
\titleformat{\subsection}[runin]
  {\normalfont\bfseries}{\thesubsection}{1em}{}

\newtheorem{definition}{\textbf{Definition}}[section]

\newtheorem{lemma}[definition]{\textbf{Lemma}}
\newtheorem{proposition}[definition]{\textbf{Proposition}}

\newtheorem{remark}[definition]{\textbf{Remark}}
\newtheorem{letterthm}{Theorem}

\newtheorem*{conjecture}{\textbf{Conjecture}}

\def\wx{\infty}

\def\car{\curvearrowright}
\def\supp{\mathrm{supp \, }}

\def\prob{\mathrm{Prob}}

\def\N {\mathbb{N}}
\def\R {\mathbb{R}}
\def\C {\mathbb{C}}

\newcommand{\rt}{\rtimes}

\newcommand{\EP}{E}

\newcommand{\Proj}{\operatorname{Proj}}

\newcommand{\cO}{\mathcal O}

\title[On invariant subalgebras of noncommutative Poisson boundaries]{On invariant subalgebras of noncommutative Poisson boundaries for higher rank lattices}

\author{Shuoxing Zhou}
\address{\'Ecole Normale Sup\'erieure\\ D\'epartement de math\'ematiques et applications\\ 45 rue d'Ulm\\ 75230 Paris Cedex 05\\ FRANCE}
\email{shuoxing.zhou@ens.psl.eu}

\begin{document}
\begin{abstract}
Let $G$ be a real connected semisimple Lie group with trivial center, no non-trivial compact factors, and all simple factors of real rank at least two. Let $\Gamma<G$ be an irreducible lattice and $P<G$ be a minimal parabolic subgroup. Amrutam--Hartman conjectured that every $\Gamma$-invariant von Neumann subalgebra $M\subset L^\infty(G/P,\nu_P)\rtimes \Gamma$ is of the form $L^\infty(G/Q,\nu_Q)\rtimes\Lambda$. We prove this classification whenever either $M\cap L^\infty(G/P,\nu_P)\not=\C 1$ or $M\cap L(\Gamma)\not =\C 1$, thereby reducing the Amrutam--Hartman conjecture to the extreme case $M\cap L^\infty(G/P,\nu_P)=M\cap L(\Gamma)=\C1$.
\end{abstract}

\maketitle

\section{Introduction}\label{introduction}

Let $G=\prod_{i=1}^d G_i$ be a real connected semisimple Lie group with trivial center, no non-trivial compact factors, and $\mathrm{rank}_{\R}(G_i)\geq 2$ for each simple factor $G_i$. Let $\Gamma<G$ be an irreducible lattice. Let $P<G$ be a minimal parabolic subgroup, let $K<G$ be a maximal compact subgroup with $G=KP$, and let $\nu_P\in\prob(G/P)$ be the unique $K$-invariant probability measure on $G/P$. Furstenberg \cite{Fur67} proved that there exists a fully supported measure $\mu\in\prob(\Gamma)$ such that the $(\Gamma,\mu)$-Poisson boundary is exactly $(G/P,\nu_P)$.

The noncommutative Poisson boundary $L^\wx(G/P,\nu_P)\rtimes \Gamma$ is closely related to Connes' rigidity conjecture for higher rank lattices; see, for example, \cite{Hou24}. Thus it is natural to study the algebraic structure of $L^\wx(G/P,\nu_P)\rtimes \Gamma$.

In this paper, we study von Neumann subalgebras $M$ of $L^\wx(G/P,\nu_P)\rtimes \Gamma$ which are invariant under the conjugation action of $\Gamma$, i.e., $u_gMu_g^*=M$ for all $g\in\Gamma$, where $u_g$ is the canonical group unitary corresponding to $g$.

Margulis' factor theorem \cite[Theorem IV.2.11]{Mar91} shows that any $\Gamma$-invariant von Neumann subalgebra of $L^\wx(G/P,\nu_P)$ is of the form $L^\wx(G/Q,\nu_Q)$ for some $P\leq Q\leq G$.

It was proved in \cite{BH19}, as the noncommutative Nevo--Zimmer theorem, that any ergodic $\Gamma$-von Neumann algebra $M$ with a normal $\Gamma$-equivariant ucp map $\Phi:M\to L^\wx(G/P,\nu_P)$ must satisfy that either $\Phi$ is $\Gamma$-invariant, or there exists a normal embedding $\iota:L^\wx(G/Q,\nu_Q)\hookrightarrow M$ for some proper parabolic subgroup $P\leq Q\lneq G$ such that $\Phi\circ\iota$ is the canonical embedding of $L^\wx(G/Q,\nu_Q)$ into $L^\wx(G/P,\nu_P)$.

It was proved in \cite{KP23} that every $\Gamma$-invariant von Neumann subalgebra $M\subset L(\Gamma)$ is of the form $M=L(\Lambda)$ for some normal subgroup $\Lambda\lhd \Gamma$.

It was proved in \cite{AH24} that, for every proper parabolic subgroup $P\leq Q\lneq G$, every $\Gamma$-invariant intermediate subalgebra
$$
  L^\wx(G/Q,\nu_Q)\subset M\subset L^\wx(G/Q,\nu_Q)\rtimes \Gamma
$$
is of the form $M=L^\wx(G/Q,\nu_Q)\rtimes \Lambda$ for some normal subgroup $\Lambda\lhd \Gamma$.

Both the results in \cite{KP23} and \cite{AH24} rely essentially on the noncommutative Nevo--Zimmer theorem \cite{BH19}. In particular, their arguments require the higher rank assumption of each simple factor of $G$ and cannot be weakened to allow simple factors of rank one.

It was proved in \cite{BH22} that every intermediate subalgebra
$$
  L(\Gamma)\subset M\subset L^\wx(G/P,\nu_P)\rtimes \Gamma
$$
is of the form $M=L^\wx(G/Q,\nu_Q)\rtimes \Gamma$ for some parabolic subgroup $P\leq Q\leq G$. Here the proof essentially relies on results from \cite{BBHP21,BBH22}. Hence the higher rank assumption in \cite{BH22} can be weakened to $G$ being a semisimple algebraic group whose total rank, instead of the rank of each simple factor, is at least two.

Summarizing the results above, in \cite{AH24}, Amrutam--Hartman also conjectured that every $\Gamma$-invariant von Neumann subalgebra $M\subset L^\wx(G/P,\nu_P)\rtimes \Gamma$ is of the form $M=L^\wx(G/Q,\nu_Q)\rtimes \Lambda$ for some parabolic subgroup $P\leq Q\leq G$ and some normal subgroup $\Lambda\lhd \Gamma$. Our main result proves this classification whenever $M\cap L(\Gamma)$ or $M\cap L^\infty(G/P,\nu_P)$ is non-trivial.

\begin{letterthm}\label{thmA}
Let $G$, $\Gamma$ and $P$ be as above. Let $M\subset L^\infty(G/P,\nu_P)\rtimes \Gamma$ be a $\Gamma$-invariant von Neumann subalgebra. Assume that either
\begin{itemize}
  \item $M\cap L(\Gamma)\not= \C 1$;
  \item or $M\cap L^\infty(G/P,\nu_P)\not= \C 1$.
\end{itemize}
Then there exist a parabolic subgroup $P\leq Q\leq G$ and a normal subgroup $\Lambda\lhd \Gamma$ such that
$$
  M=L^\infty(G/Q,\nu_Q)\rtimes \Lambda .
$$
\end{letterthm}
Let $E:L^\wx(G/P,\nu_P)\rtimes \Gamma\to L^\wx(G/P,\nu_P)$ be the canonical $\Gamma$-equivariant conditional expectation. The following lemma from \cite{AH24} reduces the proof of Theorem~\ref{thmA} to two cases: $E(M)=\C 1$, and $E(M)=L^\wx(G/Q,\nu_Q)$ for some proper parabolic subgroup $P\leq Q\lneq G$.
\begin{lemma}[ {\cite[Lemma 5.5]{AH24}}]\label{Lemma AH}
Let $M\subset L^\wx(G/P,\nu_P)\rtimes \Gamma$ be a $\Gamma$-invariant von Neumann subalgebra. Then we always have
$$E(M)\subset M.$$
In particular,
$$E(M)=M\cap L^\wx(G/P,\nu_P)$$
is a $\Gamma$-invariant von Neumann subalgebra of the form $E(M)=L^\wx(G/Q,\nu_Q)$.
\end{lemma}

For the case $E(M)=\C 1$, if $M\cap L(\Gamma)\not= \C 1$, then the result of \cite{KP23} gives a nontrivial normal subgroup $\Lambda\lhd\Gamma$ such that $M\cap L(\Gamma)=L(\Lambda)$. The following theorem settles the branch.

\begin{letterthm}\label{E(M)=C nontrivial}
Let $M\subset L^\infty(G/P)\rtimes\Gamma$ be a $\Gamma$-invariant von Neumann subalgebra such that
$$
  E(M)=\C1
$$
and
$$
  M\cap L(\Gamma)=L(\Lambda)\neq\C1
$$
for a non-trivial normal subgroup $\{e\}\not=\Lambda\lhd\Gamma$. Then 
$$
  M=L(\Lambda).
$$
\end{letterthm}
The proof uses metric ergodicity of $\Gamma\curvearrowright (G/P,\nu_P)$, together with the finite-index conclusion $[\Gamma:\Lambda]<\infty$ from Margulis' normal subgroup theorem \cite{Mar91}.

For the case that $E(M)=L^\wx(G/Q,\nu_Q)$ is non-trivial, we prove the following Theorem \ref{E(M) not C}, showing that $M$ is actually a crossed product $L^\wx(G/Q,\nu_Q)\rtimes \Lambda$ for some normal subgroup $\Lambda\lhd \Gamma$.
\begin{letterthm}\label{E(M) not C}
Assume that $$E(M)=L^\infty(G/Q,\nu_Q)=M\cap L^\wx(G/P,\nu_P)$$ for a proper parabolic subgroup $P\leq Q\lneq G$. Then we must have
$$
  M=L^\infty(G/Q,\nu_Q)\rt\Lambda
$$
for some normal subgroup $\Lambda\lhd\Gamma$.
\end{letterthm}

For the proof of Theorem \ref{E(M) not C}, we first apply the essential freeness of the action $\Gamma\curvearrowright (G/Q,\nu_Q)$ to show that for any $x\in M$ and $s\in\Gamma$, we also have $E(xu_s^*)u_s\in M$. Therefore, we have $M=\overline{\oplus_{s\in\Lambda} W_s u_s}^{\mathrm{w}}$, where
$$W_s:=\{b\in L^\infty(G/P,\nu_P)\mid bu_s\in M\}$$
is the space of $s$-Fourier coefficients. We then apply \cite[Theorem 1.2]{AH24} and the double ergodicity of the action $\Gamma\curvearrowright (G/P,\nu_P)$ to show that each $W_s$ is either $0$ or $L^\infty(G/Q,\nu_Q)$.
Moreover, $\Lambda:=\{s\in\Gamma\mid W_s=L^\infty(G/Q,\nu_Q)\}$ gives a normal subgroup of $\Gamma$. Therefore, $$M=\overline{\oplus_{s\in\Lambda} W_s u_s}^{\mathrm{w}}=L^\infty(G/Q,\nu_Q)\rt\Lambda.$$

Summarizing the results above, we reduce the Amrutam--Hartman conjecture to the following extreme case:

\begin{conjecture}
Let $M\subset L^\infty(G/P,\nu_P)\rtimes \Gamma$ be a $\Gamma$-invariant von Neumann subalgebra. Assume that
$$E(M)=M\cap L^\infty(G/P,\nu_P)=M\cap L(\Gamma)=\C 1.$$
Then we must have
$$M=\C 1.$$
\end{conjecture}

\subsection*{Acknowledgements.}

The author would like to thank Cyril Houdayer, Yongle Jiang, and Tattwamasi Amrutam for many valuable comments on this paper. The author is especially grateful to Narutaka Ozawa for pointing out a misuse of double metric ergodicity in an earlier version of this paper. The author also thanks Jianqiao Shang for helpful discussions on the use of AI.

\subsection*{AI use statement.}

ChatGPT 5.5/5.6 and Codex were used during the preparation of this paper. In particular, under the author's guidance, these tools contributed to the exploratory stage in which the main theorem and parts of its proof strategy were first identified, and were later used for English language editing. All mathematical statements, proofs, references, and final editorial decisions were independently checked and remain the sole responsibility of the author.

\section{The case of $E(M)=\C 1$}

\begin{proof}[Proof of Theorem \ref{E(M)=C nontrivial}]
Since $\Lambda\lhd \Gamma$ is non-trivial, Margulis' normal subgroup theorem \cite{Mar91} implies that $[\Gamma:\Lambda]<+\infty$.

Write $\sigma_g(a)=u_gau_g^*$ for $g\in\Gamma$. Since $E(M)=\C1$, there is a faithful normal $\Gamma$-invariant state $\varphi$ on $M$ such that
$$
  E(a)=\varphi(a)1
  \qquad(a\in M).
$$
Fix $0
\not=z\in M$, and let
$$
  H_z=\overline{\operatorname{span}}
  \left\{
    \widehat{u_h\sigma_g(z)}\mid
    h\in\Lambda,\ g\in\Gamma
  \right\}
  \subset L^2(M,\varphi).
$$
The space $H_z$ is separable and invariant under the unitary representations
$$
  U_g\widehat a=\widehat{\sigma_g(a)}
  \quad(g\in\Gamma),
  \qquad
  L_h\widehat a=\widehat{u_h a}
  \quad(h\in\Lambda).
$$
Here the invariance under $U_g$ uses the normality of $\Lambda$.

For $a\in M$ and $s\in\Gamma$, put
$$
  d_s^a=E(u_s^*a)\in L^\infty(G/P).
$$
The crossed-product Parseval identity gives
\begin{equation*}\label{eq:scalar-parseval}
  \sum_{s\in\Gamma}(d_s^a)^*d_s^b
  =E(a^*b)
  =\varphi(a^*b)1
  \qquad(a,b\in M).
\end{equation*}
Using a countable dense subspace of $H_z$, and then discarding a $\Gamma$-invariant null set of $G/P$, we may therefore define an isometry
\begin{equation*}\label{eq:scalar-Tx}
  T_x:H_z\to\ell^2(\Gamma),
  \qquad
  T_x\widehat a=(d_s^a(x))_{s\in\Gamma},
\end{equation*}
for almost every $x\in G/P$. The field $x\mapsto T_x$ is strongly measurable. For $g\in\Gamma$, let $C_g$ denote the unitary on $\ell^2(\Gamma)$ given by
$$
  C_g\delta_s=\delta_{gsg^{-1}},
$$
For $a\in M$ and $g,s\in\Gamma$,
$$  d_{gsg^{-1}}^{\,\sigma_g(a)}=E(u_{gsg^{-1}}^*u_gau_g^*)=E(u_gu_s^*au_g^*)=\sigma_g(E(u_s^*a))=\sigma_g(d_s^a).$$
Hence for almost every $x\in G/P$, we have
\begin{equation}\label{eq:scalar-Tx-covariance}
  T_{gx}=C_gT_xU_g^*.
\end{equation}

There is also a second intertwining relation. For $h\in\Lambda$,
\begin{equation}\label{Lambda eq}
  d_s^{u_h a}
  =E(u_s^*u_h a)
  =d_{h^{-1}s}^a.
\end{equation}
Let $\lambda$ and $\rho$ denote the left and right regular representation of $\Gamma$ on
$\ell^2(\Gamma)$, and let $L(\Lambda)=\lambda(\Lambda)''\subset B(\ell^2(\Gamma))$ and $R(\Lambda)=\rho(\Lambda)''\subset B(\ell^2(\Gamma))$. Then (\ref{Lambda eq}) gives
\begin{equation}\label{eq:scalar-left-intertwining}
  T_xL_h=\lambda_h T_x
  \qquad(h\in\Lambda).
\end{equation}

Put $P_x=T_xT_x^*\in B(\ell^2(\Gamma))$. Relation~\eqref{eq:scalar-left-intertwining} implies that
$$
  P_x\in\mathcal R:= L(\Lambda)'\cap B(\ell^2(\Gamma)).
$$
Let $n=[\Gamma:\Lambda]$, then, as a left $\Lambda$-representation,
$$
  \ell^2(\Gamma)\cong\ell^2(\Lambda)\otimes\C^n \mbox{ and } B(\ell^2(\Gamma))\cong B(\ell^2(\Lambda))\,\overline\otimes\,M_n(\C).
$$
and hence
$$
  \mathcal R\cong R(\Lambda)\,\overline\otimes\,M_n(\C).
$$
Therefore, $\mathcal R$ is a finite factor. Denote its normalized trace by $\tau_{\mathcal R}$. The projection space $\Proj(\mathcal R)$, equipped with
$$
  d_2(p,q)=\|p-q\|_{2,\tau_{\mathcal R}},
$$
is a separable metric space.

By \eqref{eq:scalar-Tx-covariance},
$$
  P_{gx}=C_gP_xC_g^*.
$$
Normality of $\Lambda$ implies that $\operatorname{Ad}C_g=\operatorname{Ad}(\lambda_g\cdot \rho_g)$ preserves $L(\Lambda)$ and $\mathcal R=L(\Lambda)'$. It preserves the unique normalized trace and hence acts isometrically for $d_2$. The map $x\in G/P\mapsto P_x\in \Proj(\mathcal R)$ is measurable for the $d_2$-Borel structure: on bounded subsets of the finite von Neumann algebra $\mathcal R$, weak-operator measurability implies $L^2$-measurability. By \cite{Ka03}, metric ergodicity of the Poisson boundary $(G/P,\nu_P)$ now gives a projection $P\in\mathcal R$ such that
\begin{equation}\label{eq:scalar-P-constant}
  P_x=P
  \qquad\text{for almost every }x\in G/P.
\end{equation}
After taking a common conull set in \eqref{eq:scalar-Tx-covariance}, we also have
\begin{equation}\label{eq:scalar-P-invariant}
  C_gPC_g^*=P
  \qquad(g\in\Gamma).
\end{equation}

Let \(K=P\ell^2(\Gamma)\).  By \eqref{eq:scalar-P-constant}, for almost every
\(x\), the isometry \(T_x\) is a unitary from \(H_z\) onto \(K\).  Fix one such
unitary
\[
  T_0:H_z\to K
\]
satisfying \eqref{eq:scalar-left-intertwining}.  Put
\[
  \mathcal R_P=P\mathcal RP .
\]
Then \(\mathcal R_P\) is also a finite factor.  Denote its normalized trace by
\(\tau_P\).

Consider the space
\[
  \mathcal I
  =
  \{T:H_z\to K\text{ unitary}\mid TL_h=\lambda_hT\text{ for all }h\in\Lambda\}.
\]
Via
\[
  T\longmapsto TT_0^*,
\]
we can identify \(\mathcal I\) with \(\mathcal U(\mathcal R_P)\).  We equip
\(\mathcal I\) with the corresponding metric
\[
  d(T,T')
  =
  \|TT_0^*-T'T_0^*\|_{2,\tau_P}.
\]
This makes $\mathcal I$ a separable metric space.

The covariance relation \eqref{eq:scalar-Tx-covariance} defines an action of
\(\Gamma\) on \(\mathcal I\) by
\[
  g\cdot T=C_gTU_g^* .
\]
Indeed, for \(h\in\Lambda\),
\[
  (C_gTU_g^*)L_h
  =
  C_gT L_{g^{-1}hg}U_g^*
  =
  C_g\lambda_{g^{-1}hg}TU_g^*
  =
  \lambda_h C_gTU_g^*,
\]
where normality of \(\Lambda\) is used.  Under the above identification
\(\mathcal I\simeq\mathcal U(\mathcal R_P)\), this action is given by
\[
  TT_0^*
  \longmapsto
  C_gTT_0^*V_g^*,
  \qquad
  V_g=T_0U_gT_0^* .
\]
Left multiplication by \(C_g\) and right multiplication by \(V_g^*\) preserve
the \(L^2\)-metric on the finite factor \(\mathcal R_P\), so the action is
isometric.

The map \(x\mapsto T_x\) is measurable as a map into \(\mathcal I\), and
\eqref{eq:scalar-Tx-covariance} says that it is \(\Gamma\)-equivariant.
Metric ergodicity of \((G/P,\nu_P)\) therefore gives a fixed
\(T\in\mathcal I\) such that
\[
  T_x=T
  \qquad\text{for almost every }x\in G/P.
\]

Since \(\widehat z\in H_z\), it follows that, for every \(s\in\Gamma\),
\[
  d_s^z(x)=(T_x\widehat z)(s)=(T\widehat z)(s)
\]
is almost everywhere constant.  Hence \(d_s^z\in\C1\) for every \(s\in\Gamma\).
Equivalently, all Fourier coefficients of \(z\) are scalar, and therefore
\(z\in L(\Gamma)\).  Since \(z\in M\) was arbitrary, we obtain
\[
  M\subset L(\Gamma).
\]
Together with \(M\cap L(\Gamma)=L(\Lambda)\), this proves \(M=L(\Lambda)\).
\end{proof}

\begin{remark}
Independently, Tattwamasi Amrutam and Yongle Jiang have obtained a related result by different methods. For pmp mixing actions $\Gamma\curvearrowright (X,\nu)$ of a class of ICC groups including torsion-free hyperbolic groups, they prove that if $M\subset L^\infty(X,\nu)\rtimes\Gamma$ is a $\Gamma$-invariant von Neumann subalgebra and $E(M)=\C 1$, then $M\subset L(\Gamma)$.
\end{remark}

\section{The case of $E(M)=L^\wx(G/Q)$}
Throughout this section, we let $G$, $\Gamma$ and $P$ be as in Section \ref{introduction}, and let $M\subset L^\infty(G/P)\rtimes \Gamma$ be a $\Gamma$-invariant von Neumann subalgebra with $$E(M)=L^\infty(G/Q)=M\cap L^\wx(G/P)$$ for a proper parabolic subgroup $P\leq Q\lneq G$ as in Theorem~\ref{E(M) not C}.

Let $\pi=\pi_Q:G/P\to G/Q$ be the factor map. We identify $L^\infty(G/Q)$ with $\pi^*L^\infty(G/Q)\subset L^\infty(G/P)$. We still denote the action of $g\in\Gamma$ on $L^\infty(G/P)\rtimes\Gamma$ by $\sigma_g(a)=u_ga u_g^*$ for $a\in L^\infty(G/P)\rtimes\Gamma$.

\begin{lemma}\label{lem:slice-relations}
For every $x\in M$ and $s\in\Gamma$,
$$
  \EP(xu_s^*)u_s\in M.
$$
Define
$$
  W_s=\{b\in L^\infty(G/P)\mid bu_s\in M\}.
$$
Then, for all $g,s,t\in\Gamma$,
\begin{itemize}
  \item [(1)] $W_e=L^\infty(G/Q)$,
  \item [(2)] $L^\infty(G/Q) W_s L^\infty(G/Q)\subset W_s$,
  \item [(3)] $W_{s^{-1}}=\sigma_{s^{-1}}(W_s^*)$,
  \item [(4)] $\sigma_g(W_s)=W_{gsg^{-1}}$,
  \item [(5)] $W_s\sigma_s(W_t)\subset W_{st}$,
  \item [(6)] $W_sW_s^*+W_s^*W_s\subset L^\infty(G/Q)$.
\end{itemize}

\end{lemma}

\begin{proof}
Fix $s\in\Gamma$. Let $K\subset\Gamma\setminus \{s\}$ be finite and put
$$
  F_K=\{ts^{-1}\mid t\in K\}\subset\Gamma\setminus\{e\}.
$$
By \cite[Lemma 6.2]{BH19} (see also \cite[Remark 13]{Oz16}), the action $\Gamma\curvearrowright G/Q$ is essentially free. Then \cite[Lemma 3.1]{Suz18} gives a finite measurable partition $G/Q=\bigsqcup_i\Omega_i$ such that
$$
  \Omega_i\cap r\Omega_i
$$
is null for every $r\in F_K$ and every $i$. Let $p_i=\mathds{1}_{\Omega_i}\in L^\infty(G/Q)$. Then $\sum_i p_i=1$ and $p_i\sigma_r(p_i)=0$ for all $r\in F_K$.

For $a\in L^\infty(G/P)\rt \Gamma$, define $T_K^{(s)}: L^\infty(G/P)\rtimes\Gamma\to L^\infty(G/P)\rtimes\Gamma$ by
$$
  T_K^{(s)}(a)
  =
  \sum_i p_i a\sigma_{s^{-1}}(p_i)
  =
  \Bigl(\sum_i p_i(au_s^*)p_i\Bigr)u_s .
$$
The map $R_{u_s^*}:a\mapsto au_s^*$, the ucp map $E_K:y\mapsto\sum_i p_iyp_i$, and $R_{u_s}:a\mapsto au_s$ are all normal contractions, hence so is $T_K^{(s)}=R_{u_s}\circ E_K\circ R_{u_s^*}$. Let now $x\in M$. Since $p_i$ and $\sigma_{s^{-1}}(p_i)$ belong to $L^\infty(G/Q)\subset M$, we have
$$
  T_K^{(s)}(x)=\sum_i p_ix\sigma_{s^{-1}}(p_i)\in M.
$$

For $b\in L^\infty(G/P)$ and $t\in\Gamma$, we have
\begin{equation}\label{T_K(bus)}
\begin{aligned}
T_K^{(s)}(bu_t)
 &= \sum_i p_i b u_t \sigma_{s^{-1}}(p_i)= \sum_i p_i b \sigma_{ts^{-1}}(p_i)u_t =
 \begin{cases}
 bu_s, & \text{if } t=s,\\
 0, & \text{if }t\in K.
 \end{cases}
\end{aligned}
\end{equation} 

Let $K_n\subset \Gamma\setminus\{s\}$ be an increasing sequence of finite subsets with $\bigcup_n K_n=\Gamma\setminus\{s\}$. Take $\omega\in\beta\N\setminus\N$ and define $T^{(s)}:L^\infty(G/P)\rtimes\Gamma\to L^\infty(G/P)\rtimes\Gamma$ by
$$
  T^{(s)}(a)=\text{wo-}\lim_{n\to\omega}T_{K_n}^{(s)}(a)
  \qquad(a\in L^\infty(G/P)\rtimes\Gamma).
$$
Let $\varphi$ be the normal faithful state associated with $\nu_P$ on $L^\infty(G/P,\nu_P)\rtimes\Gamma$. Then
$$R_{u_s^*}\circ T^{(s)}\circ R_{u_s}=\lim_{\omega} R_{u_s^*}\circ T_{K_n}^{(s)}\circ R_{u_s}=\lim_{\omega} E_{K_n},$$
which is a ucp map preserving the normal faithful state $\varphi$, as each map $E_K:y\mapsto\sum_i p_iyp_i$ is. Hence $R_{u_s^*}\circ T^{(s)}\circ R_{u_s}$ is normal and so is $T^{(s)}$.

For any $a\in L^\infty(G/P)\rtimes\Gamma$, by (\ref{T_K(bus)}) and the normality of $T^{(s)}$, we have
$$
  T^{(s)}(a)=\EP(au_s^*)u_s.
$$

Let now $x\in M$. Since $T^{(s)}_{K}(x)\in M$ for any finite $K\subset\Gamma\setminus\{s\}$, we have
$$
  \EP(xu_s^*)u_s=T^{(s)}(x)=\text{wo-}\lim_{n\to\omega}T_{K_n}^{(s)}(x)\in M.
$$

It remains to check the listed algebraic relations (1-6). First
$$
  W_e=M\cap L^\infty(G/P)=L^\infty(G/Q).
$$
Let $a,c\in L^\infty(G/Q)$ and $b\in W_s$, then
$$
  a(bu_s)c=ab\sigma_s(c)u_s\in M,
$$
so
$$
  L^\infty(G/Q) W_s L^\infty(G/Q)\subset W_s.
$$
Since $bu_s\in M$ if and only if $(bu_s)^*=\sigma_{s^{-1}}(\bar{b}) u_{s^{-1}}\in M$, we have
$$W_{s^{-1}}=\sigma_{s^{-1}}(W_s^*).$$
Let $b\in W_s$, then
$$
  u_g(bu_s)u_g^*=\sigma_g(b)u_{gsg^{-1}}\in M,
$$
which gives $\sigma_g(W_s)=W_{gsg^{-1}}$. Let $b\in W_s$ and $c\in W_t$, then
$$
  (bu_s)(cu_t)=b\sigma_s(c)u_{st}\in M,
$$
so $b\sigma_s(c)\in W_{st}$. Finally, for $b,c\in W_s$,
$$
  (bu_s)(cu_s)^*=b\overline c\in
  M\cap L^\infty(G/P)=L^\infty(G/Q)
$$
and
$$
  (bu_s)^*(cu_s)=\sigma_{s^{-1}}(\overline b c)\in
  M\cap L^\infty(G/P)=L^\infty(G/Q).
$$
Applying $\sigma_s$ to the last inclusion gives $\overline b c\in L^\infty(G/Q)$, and the lemma follows.
\end{proof}

\begin{proposition}\label{prop:rank-one-slices}
For every $s\in\Gamma$ there exists a projection $p_s\in\Proj(L^\infty(G/Q))$ and a partial isometry $v_s\in L^\infty(G/P)$ such that
$$
  v_s^*v_s=v_sv_s^*=p_s,\qquad W_s=L^\infty(G/Q) p_s v_s.
$$
They may be chosen with $p_e=1$ and $v_e=1$. Equivalently, for every $b\in W_s$ there is a unique $a_b\in L^\infty(G/Q) p_s$ such that $b=a_bv_s$.
\end{proposition}

\begin{proof}
Fix $s\in\Gamma$. Then $W_s$ is ultraweakly closed in $L^\infty(G/P)$. It is also a bimodule over $L^\infty(G/Q)$ and satisfies
$$
  b\overline c,\ \overline b c\in L^\infty(G/Q)\qquad(b,c\in W_s).
$$
Since $(G/P,\nu_P)$ is a standard probability space, $L^1(G/P,\nu_P)$ is separable. Hence the unit ball of $W_s$ is compact metrizable for the ultraweak topology. Choose a sequence $(b_n)_{n\ge1}$ in the unit ball of $W_s$ which is ultraweakly dense. Put
$$
  r_n=\mathds{1}_{\supp(|b_n|)}\in\Proj(L^\infty(G/Q)),
$$
and define pairwise orthogonal projections
$$
  q_1=r_1,\qquad
  q_n=r_n\Bigl(1-\bigvee_{m<n}r_m\Bigr)\quad(n\ge2).
$$
Let
$$
  p_s=\bigvee_{n\ge1}q_n=\bigvee_{n\ge1}r_n\in\Proj(L^\infty(G/Q))
$$
and
$$
  v_s=\sum_{n\ge1}q_n\,b_n\,|b_n|^{-1},
$$
where $|b_n|^{-1}$ is taken only on $\supp(|b_n|)$ and the summands are zero off $q_n$. The summands are disjoint, so $v_s\in L^\infty(G/P)$ and
$$
  |v_s|=p_s,\qquad v_s^*v_s=v_sv_s^*=p_s.
$$

We first show that $v_s\in W_s$. For $k\ge1$, set
$$
  q_{n,k}=q_n\,\mathds{1}_{\{|b_n|\ge 1/k\}}\in L^\infty(G/Q).
$$
Then
$$
  q_{n,k}b_n|b_n|^{-1}=(q_{n,k}|b_n|^{-1})b_n\in L^\infty(G/Q) W_s \subset W_s,
$$
and, as $k\to\infty$, these elements converge ultraweakly to $q_nb_n|b_n|^{-1}=q_nv_s$. Hence $q_nv_s\in W_s$. Finite sums $\sum_{n\le N}q_nv_s$ belong to $W_s$ and converge ultraweakly to $v_s$, so $v_s\in W_s$. Consequently $L^\infty(G/Q) p_s v_s\subset W_s$.

Conversely, let $b\in W_s$. Since $(1-p_s)b_n=0$ for every $n$, ultraweak density gives $(1-p_s)b=0$, so $b=bp_s$. On $q_n$,
$$
  q_n b\overline{v_s}=q_n\,b\,\overline{b_n}\,|b_n|^{-1}.
$$
For $k\ge1$, the truncated function
$$
  q_{n,k}\,b\,\overline{b_n}\,|b_n|^{-1}
$$
belongs to $L^\infty(G/Q)$, since $b\overline{b_n}\in L^\infty(G/Q)$ and $q_{n,k}|b_n|^{-1}\in L^\infty(G/Q)$ is bounded. These functions converge ultraweakly, bounded by $\|b\|_\infty q_n$, to $q_nb\overline{v_s}$. Thus $q_nb\overline{v_s}\in L^\infty(G/Q) q_n$, and since the $q_n$'s are disjoint,
$$
  a_b:=b\overline{v_s}=\sum_n q_nb\overline{v_s}
$$
belongs to $L^\infty(G/Q) p_s$. Finally
$$
  a_bv_s=b\overline{v_s}v_s=bp_s=b.
$$
This proves $W_s=L^\infty(G/Q) p_sv_s$. Uniqueness follows by multiplying $av_s=a'v_s$ by $\overline{v_s}$. For $s=e$, take $p_e=1$ and $v_e=1$.
\end{proof}

\begin{lemma}\label{lem:support-relations}
The projections from Proposition~\ref{prop:rank-one-slices} satisfy, for all $g,s,t\in\Gamma$,
$$
  p_e=1,\qquad
  p_{s^{-1}}=\sigma_{s^{-1}}(p_s),\qquad
  p_{gsg^{-1}}=\sigma_g(p_s),
$$
and
$$
  p_s\sigma_s(p_t)\leq p_{st}.
$$
\end{lemma}

\begin{proof}
The projection $p_s$ is the least projection $p\in L^\infty(G/Q)$ such that $W_s=pW_s$. Hence it is determined uniquely by $W_s$. The inverse and conjugacy relations in Lemma~\ref{lem:slice-relations} immediately give
$$
  p_{s^{-1}}=\sigma_{s^{-1}}(p_s)
  \quad\text{and}\quad
  p_{gsg^{-1}}=\sigma_g(p_s).
$$
Also $p_e=1$ because $W_e=L^\infty(G/Q)$.

Since $p_sv_s\in W_s$ and $p_tv_t\in W_t$, the multiplication relation gives
$$
  p_sv_s\,\sigma_s(p_tv_t)\in W_{st}.
$$
The absolute value of the function on the left is $p_s\sigma_s(p_t)$. Every element of $W_{st}=L^\infty(G/Q)p_{st}v_{st}$ is supported by $p_{st}$, and therefore $p_s\sigma_s(p_t)\leq p_{st}$.
\end{proof}

\begin{proposition}\label{prop:support-rigidity}
There exists a normal subgroup $\Lambda\lhd\Gamma$ such that
$$
  p_s=
  \begin{cases}
    1,&s\in\Lambda,\\
    0,&s\notin\Lambda.
  \end{cases}
$$
\end{proposition}

\begin{proof}
Define
$$
  \widetilde M=\overline{\oplus_{s\in\Gamma}L^\infty(G/Q)p_su_s}^{\mathrm{w}}
  =\overline{\operatorname{span}}^{\,\mathrm{w}}
  \{bp_su_s\mid b\in L^\infty(G/Q),\ s\in\Gamma\}
  \subset L^\infty(G/Q)\rtimes\Gamma.
$$
For $b,c\in L^\infty(G/Q)$ and $s,t\in\Gamma$, Lemma~\ref{lem:support-relations} gives
$$
  (bp_su_s)(cp_tu_t)
  =b\sigma_s(c)p_s\sigma_s(p_t)u_{st}
  \in L^\infty(G/Q)p_{st}u_{st}
$$
and
$$
  (bp_su_s)^*
  =\sigma_{s^{-1}}(p_sb^*)u_{s^{-1}}
  \in L^\infty(G/Q)p_{s^{-1}}u_{s^{-1}}.
$$
Thus $\widetilde M$ is a von Neumann algebra, and $p_e=1$ gives $L^\infty(G/Q)\subset\widetilde M$. Moreover,
$$
  u_g(bp_su_s)u_g^*
  =\sigma_g(b)p_{gsg^{-1}}u_{gsg^{-1}},
$$
so $\widetilde M$ is $\Gamma$-invariant.

Hence \cite[Theorem 1.2]{AH24}, applied to
$$
  L^\infty(G/Q)\subset\widetilde M\subset L^\infty(G/Q)\rtimes\Gamma,
$$
gives a normal subgroup $\Lambda\lhd\Gamma$ such that
$$
  \widetilde M=L^\infty(G/Q)\rtimes\Lambda.
$$

Consequently, $L^\infty(G/Q)p_s=L^\infty(G/Q)$ for $s\in\Lambda$ and $L^\infty(G/Q)p_s=0$ for $s\notin\Lambda$, which proves the assertion.
\end{proof}

Let $\pi:G/P\to G/Q$ be the factor map. Let
$$
  Z=G/P\times_{G/Q}G/P:=\{(x_1,x_2)\in G/P\times G/P\mid \pi(x_1)=\pi(x_2)\in G/Q\}.
$$
Equip $Z$ with the standard relative product measure over $G/Q$. Namely, let
$$
  \nu_P=\int_{G/Q}\nu_y\,d\nu_Q(y)
$$
be the disintegration of $\nu_P$ over $\nu_Q$ for $\pi:G/P\to G/Q$, where $\nu_y$ is supported on $\pi^{-1}(y)\cong Q/P$. Then
$$
  \nu_Z=\nu_P\times_{\nu_Q}\nu_P
  :=
  \int_{G/Q}\nu_y\otimes\nu_y\,d\nu_Q(y)
$$
is the measure with respect to which all almost-everywhere statements on $Z$ are understood.

Then for $s\in \Lambda$, to prove $W_s= L^\infty(G/Q)$, it is enough to show that $v_s\in \mathcal{U}(L^\infty(G/Q))$, i.e., for almost every $(x_1,x_2)\in G/P\times_{G/Q} G/P=Z$, one has $v_s(x_1)=v_s(x_2)$.

For $(x_1,x_2)\in Z=G/P\times_{G/Q}G/P$, set
$$
  D_s(x_1,x_2)=v_s(x_1)\overline{v_s(x_2)}.
$$

\begin{lemma}\label{lem:relative-bruhat-factor}
There exists a measure-class preserving $G$-factor map
$$
  \theta:(G/P\times G/P,[\nu_P\times\nu_P])\longrightarrow (Z,[\nu_Z]).
$$
\end{lemma}

\begin{proof}
We construct the map on conull open orbits. Let $w_0$ be the longest Weyl element of $G$, and let $w_Q$ be the longest Weyl element of the Levi Weyl subgroup associated with $Q$. By the Bruhat decomposition,
$$
  \cO=G\cdot(P,w_0P)\subset G/P\times G/P
$$
is the open conull $G$-orbit; see \cite[Theorem 7.40 and Proposition 8.45]{Knapp2002}. Similarly, applying the Bruhat decomposition inside the Levi factor of $Q$, the relative product
$$
  Z=G/P\times_{G/Q}G/P\cong G\times_Q(Q/P\times Q/P)
$$
has an open conull $G$-orbit. Here
$$
  G\times_Q(Q/P\times Q/P):=(G\times(Q/P\times Q/P))/Q,
$$
where $Q$ acts on the right by
$$
  (g,y_1,y_2)q=(gq,q^{-1}y_1,q^{-1}y_2),
$$
and the identification with $Z$ is given by
$$
  [g,y_1,y_2]\longmapsto (gy_1,gy_2).
$$
Thus this open conull orbit is
$$
  \cO_Q=G\cdot(P,w_QP).
$$

Let
$$
  H_0=P\cap w_0Pw_0^{-1},\qquad H_Q=P\cap w_QPw_Q^{-1}.
$$
Then
$$
  \cO\simeq G/H_0,\qquad \cO_Q\simeq G/H_Q.
$$
If $P=M_0A_0N_0$ is a Langlands decomposition of the minimal parabolic, then
$$
  H_0=P\cap w_0Pw_0^{-1}=M_0A_0\subset P\cap w_QPw_Q^{-1}=H_Q.
$$
Hence the quotient map
$$
  G/H_0\longrightarrow G/H_Q,\qquad gH_0\longmapsto gH_Q,
$$
is well defined and $G$-equivariant. Under the above orbit identifications it is
$$
  \theta(gP,gw_0P)=(gP,gw_QP).
$$
The push-forward of the homogeneous smooth measure class on $G/H_0$ is the homogeneous smooth measure class on $G/H_Q$. These measure classes agree with $[\nu_P\times\nu_P]$ on $\cO$ and with $[\nu_Z]$ on $\cO_Q$, respectively. Since $\cO$ and $\cO_Q$ are conull, extending $\theta$ arbitrarily on the null complement of $\cO$ gives a measure-class preserving $G$-factor map
$$
  \theta:(G/P\times G/P,[\nu_P\times\nu_P])\to (Z,[\nu_Z])
$$
in the usual measure-algebra sense.
\end{proof}

\begin{proposition}\label{prop:phase-descent}
For every $s\in\Lambda$, we have $v_s\in \mathcal{U}(L^\infty(G/Q))$. Consequently,
$$
  W_s=L^\infty(G/Q)\quad(s\in\Lambda),
  \qquad
  W_s=0\quad(s\notin\Lambda).
$$
\end{proposition}

\begin{proof}
If $\Lambda=\{e\}$, the assertion follows from $v_e=1$. Assume that $\Lambda\neq\{e\}$. Then $\Lambda$ has finite index in $\Gamma$.

For every $s\in\Lambda$, Proposition~\ref{prop:support-rigidity} gives $p_s=1$. Hence Proposition~\ref{prop:rank-one-slices} gives
$$
  v_s\in\mathcal U(L^\infty(G/P)),
  \qquad
  W_s=L^\infty(G/Q)v_s.
$$

For $s,t\in\Lambda$, the multiplication relation $W_s\sigma_s(W_t)\subset W_{st}$ gives a unitary $\omega_{s,t}\in\mathcal U(L^\infty(G/Q))$ such that
\begin{equation}\label{vv=v}
  v_s\sigma_s(v_t)=\omega_{s,t}v_{st}.
\end{equation}

Likewise, for $g\in\Gamma$ and $s\in\Lambda$, normality of $\Lambda$ and the conjugacy relation $\sigma_g(W_s)=W_{gsg^{-1}}$ give a unitary $\eta_{g,s}\in\mathcal U(L^\infty(G/Q))$ such that
\begin{equation}\label{v=v}
  \sigma_g(v_s)=\eta_{g,s}v_{gsg^{-1}}.
\end{equation}

Let $
  Z=G/P\times_{G/Q}G/P
$ and
$$
  D_s(x_1,x_2)=v_s(x_1)\overline{v_s(x_2)}
  \qquad ((x_1,x_2)\in Z).
$$
be as above. Then $D_s$ is a unitary in $L^\infty(Z)$ and $D_e=1$. By \eqref{vv=v} and \eqref{v=v}, we have that for all $s,t\in\Lambda$ and $g\in\Gamma$,
\begin{equation}\label{eq:phase-identities}
  D_{st}=D_s\sigma_s(D_t),
  \qquad
  D_{gsg^{-1}}=\sigma_g(D_s).
\end{equation}
where $\sigma_g(F)(x_1,x_2)=F(g^{-1}x_1,g^{-1}x_2)$.

Fix $r\in\Lambda$. For $s\in\Lambda$, applying the first identity in \eqref{eq:phase-identities} to the two decompositions
$$
  rs=r\cdot s=(rsr^{-1})\cdot r
$$
and using the second identity gives
$$
  D_r\sigma_r(D_s)
  =D_{rs}
  =D_{rsr^{-1}}\sigma_{rsr^{-1}}(D_r)
  =\sigma_r(D_s)\sigma_{rsr^{-1}}(D_r).
$$
Cancelling the unitary $\sigma_r(D_s)$ yields
$$
  D_r=\sigma_{rsr^{-1}}(D_r).
$$
Since $rsr^{-1}$ ranges over $\Lambda$, the function $D_r:Z\to \C$ is $\Lambda$-invariant.

Since $[\Gamma:\Lambda]<+\infty$, $\Lambda<G$ is also an irreducible lattice. By \cite[the proof of Proposition 2.9]{BCFS25} and \cite[Main Theorem]{BL96}, there exists a symmetric generating measure $\mu_\Lambda\in\prob(\Lambda)$ such that the $(\Lambda,\mu_\Lambda)$-Poisson boundary is exactly $(G/P,\nu_P)$. By the double ergodicity of Poisson boundaries \cite{Ka03}, $\Lambda\car G/P\times G/P$ is ergodic. Moreover, by Lemma \ref{lem:relative-bruhat-factor}, as a $\Lambda$-factor of $G/P\times G/P$, $\Lambda\curvearrowright(Z,\nu_Z)$ is also ergodic.

By the ergodicity, it follows that $D_r=c_r$ almost everywhere for some $c_r\in\mathbb T$. On the relative triple product $G/P\times_{G/Q}G/P\times_{G/Q}G/P$, we also have
$$
  D_r(x_1,x_2)D_r(x_2,x_3)=D_r(x_1,x_3).
$$
Therefore $c_r^2=c_r$, and so $c_r=1$. Thus
$$
  v_r(x_1)=v_r(x_2)
$$
for almost every pair in the same $G/Q$-fibre. By the standard fibrewise characterization of pullback functions, obtained from the above disintegration, $v_r\in\pi^*L^\infty(G/Q)=L^\infty(G/Q)$. Since $r\in\Lambda$ was arbitrary, the first assertion follows. Proposition~\ref{prop:rank-one-slices} and Proposition~\ref{prop:support-rigidity} now give the stated description of all $W_s$.
\end{proof}

We are now ready to prove Theorem~\ref{E(M) not C}.
\begin{proof}[Proof of Theorem~\ref{E(M) not C}]
Let
$$\Lambda=\{s\in\Gamma \mid W_s=L^\infty(G/Q)\}\subset\Gamma.$$
By Proposition~\ref{prop:phase-descent} and the algebraic relations (1-6) in Lemma~\ref{lem:slice-relations}, $\Lambda$ is a normal subgroup of $\Gamma$ and $W_s=0$ for any $s\in\Gamma\setminus\Lambda$.

By Lemma \ref{lem:slice-relations} and the definition of $W_s$, we have
$$L^\infty(G/Q)\rt_{\mathrm{alg}}\Lambda=\bigoplus _{s\in\Lambda} W_su_s\subset M \subset \overline{\bigoplus _{s\in\Gamma} W_su_s}^{\mathrm{w}}=L^\infty(G/Q)\rtimes\Lambda.$$
Therefore, $M=L^\infty(G/Q)\rtimes\Lambda$.
\end{proof}

\section{Proof of the main result}
\begin{proof}[Proof of Theorem~\ref{thmA}]

Assume that $M\cap L(\Gamma)\not =\C 1$ or $M\cap L^\infty(G/P)\not =\C 1$. Let
$$
  E:L^\infty(G/P)\rtimes \Gamma\to L^\infty(G/P)
$$
be the canonical conditional expectation. By Lemma~\ref{Lemma AH} {\cite[Lemma 5.5]{AH24}}, there exists a parabolic subgroup $Q$ with $P\leq Q\leq G$ such that
$$
  E(M)=L^\infty(G/Q)=     M\cap L^\infty(G/P).
$$

If $Q\lneq G$, then Theorem~\ref{E(M) not C} gives
$$
  M=L^\infty(G/Q)\rtimes \Lambda 
$$
for a normal subgroup $\Lambda\lhd\Gamma$.

Assume now that $Q=G$. Then
$$
  E(M)=M\cap L^\infty(G/P)=\C1.
$$
The alternative $M\cap L^\infty(G/P)\not =\C 1$ cannot hold. Hence we must have $M\cap L(\Gamma)\not =\C 1$. By \cite{KP23}, there exists a non-trivial normal subgroup $\{e\}\not=\Lambda\lhd\Gamma$ such that
$$
 M\cap L(\Gamma)=L(\Lambda).
$$
Theorem \ref{E(M)=C nontrivial} therefore gives
$$
  M=L(\Lambda),
$$
which finishes the proof.
\end{proof}

\bibliographystyle{alphaurl}
\bibliography{ref}
\end{document}